\documentclass[10pt]{amsart}
\usepackage{amssymb,amsthm,amsmath,amsfonts}
\usepackage[numbers]{natbib}
\usepackage{hyperref}
\usepackage{enumerate}

\makeatletter
\g@addto@macro\th@plain{\thm@headpunct{}}
\makeatother

\newtheorem{theorem}{Theorem}[section]

\newtheorem{lemma}[theorem]{Lemma}

\newcommand{\xx}{ {\textbf x} }
\newcommand{\ab}{ {\textbf a} }
\newcommand{\bb}{ {\textbf b} }
\newcommand{\cb}{ {\textbf c} }

\newcommand{\yy}{ {\textbf y} }

\newcommand{\zz}{ {\textbf z} }
\newcommand{\ff}{ {\textbf f} }
\newcommand{\gb}{ {\textbf g} }
\newcommand{\ee}{ {\textbf e} }

\newcommand{\ub}{ {\textbf u} }
\newcommand{\vb}{ {\textbf v} }

\newcommand{\VV}{ \Omega }
\newcommand{\RR}{\mathbb{R}}
\newcommand{\KK}{\mathbb{K}}

\newcommand{\LL}{\mathbb{L}}
\newcommand{\PP}{\mathbb{P}}
\newcommand{\En}{\mathbb{E}}

\newcommand{\tr}{\mathrm{tr}\,}
\newcommand{\Trace}{\mathrm{Trace}\,}
\newcommand{\DDet}{\mathrm{Det}}

\providecommand{\scalar}[1]{\left\langle#1\right\rangle}

\title{The Matsumoto-Yor Property and Its Converse on Symmetric Cones}
\author[B. Ko\l{}odziejek]{Bartosz Ko\l{}odziejek}
\address{Faculty of Mathematics and Information Science\\Warsaw University of Technology\\Pl. Politechniki 1\\00-661 Warsaw, Poland}
\email{kolodziejekb@mini.pw.edu.pl}



\begin{document}

\begin{abstract}
The Matsumoto--Yor (MY) property of the generalized inverse Gaussian and gamma distributions has many generalizations. As it was observed in (Letac and Weso{\l}owski in Ann Probab 28:1371--1383, 2000) the natural framework for the multivariate MY property is symmetric cones; however they prove their results for the cone of symmetric positive definite real matrices only.

In this paper, we prove the converse to the symmetric cone-variate MY property, which extends some earlier results. The smoothness assumption for the densities of respective variables is reduced to the continuity only. This enhancement was possible due to the new solution of a related functional equation for real functions defined on symmetric cones.
\end{abstract}

\maketitle

\section{Introduction}
\citet{MatYor01,MatYor03} have shown that if $X$ and $Y$ are independent random variables, $Y$ is gamma distributed with the shape parameter $p$ and the scale parameter $a$ and $X$ has the generalized inverse Gaussian distribution (GIG) with parameters $(-p,a,b)$, then the random variables $U=(X+Y)^{-1}$ and $V=X^{-1}-(X+Y)^{-1}$ are independent with respective distributions GIG with parameters $(-p,b,a)$ and gamma with parameters $p$ and $b$. 

Matsumoto and Yor asked about the converse theorem based on the independence of $U$ and $V$. Assume that $X$ and $Y$ are non-degenerate nonnegative independent random variables, such that $U$ and $V$ are independent. Does this imply that $X$ and $Y$ must follow GIG and gamma distributions, respectively?

A positive answer to this question was given by \citet{WesLe00}, with the use of Laplace transforms. In the same paper, both the Matsumoto-Yor property and its converse (with additional smoothness assumptions) were generalized to the cone $\VV_+$ of symmetric positive definite $(r,r)$ real matrices in the following way. For $p>(r-1)/2$ and $\ab, \bb\in\VV_+$, consider two independent random variables $X$ and $Y$ with following densities
\begin{align*}
\mu_{-p,\ab,\bb}(\mathrm{d}\xx)&=c_1 (\det\xx)^{-p-(r+1)/2}\exp\left(-\tr(\ab\cdot\xx)-\tr(\bb\cdot\xx^{-1})\right)I_{\VV_+}(\xx)\mathrm{d}\xx,\\
\gamma_{p,\ab}(\mathrm{d}\yy)&=c_2 (\det\yy)^{p-(r+1)/2}\exp(-\tr(\ab\cdot\yy))I_{\VV_+}(\yy)\mathrm{d}\yy.
\end{align*}
The distribution of $X$ is the GIG with parameters $(-p, \ab, \bb)$ and the distribution of $Y$ is the Wishart distribution with shape parameter $p$ and scale parameter $\ab$. Letac and Weso{\l}owski have shown that if $X$ and $Y$ are as above, then $(U,V)$ has distribution $\mu_{-p,\bb,\ab}\otimes \gamma_{p,\bb}$. As was observed by the authors, the natural framework for Matsumoto--Yor property is symmetric cones. Statement of a symmetric cone version of Matsumoto--Yor property is given in Sect. \ref{secProb}.

In this paper, we give a new proof of the converse result of the Matsumoto--Yor property, when $X$ and $Y$ take values in any irreducible symmetric cone. The smoothness assumption is reduced from $C^2$ densities in \citep{WesLe00} and differentiability in \citep{Wes02} to the continuity only. A new solution of a related functional equation on symmetric cones (see Theorem \ref{th1}) was found under the assumption of continuity of respective functions with the use of the corresponding univariate result due to \citet{Wes022}. Similar reduction in regularity assumptions was recently performed in the density version of Lukacs--Olkin--Rubin in \citep{KoloJOTP}.

It is worth mentioning several related one-dimensional results \citep{Jap04,KouVal12} as well as results for random matrices \citep{LeSym99,Kou12}.

While solving the functional equation, we use Hua's identity, which allows to write the inverse of $V=X^{-1}-(X+Y)^{-1}$ in a very convenient form:
\begin{align}\label{huavp}
V^{-1}=X+X\cdot Y^{-1}\cdot X.
\end{align}
Hua's identity has already proved to be useful in some problems related to GIG and Wishart distributions -- see \citep{Ber95}, where it was used to analyze some random continued fractions on symmetric cones. 

The paper is organized as follows. We start in the next section with some basic definitions and theorems regarding analysis on symmetric cones. In Sect. \ref{secProb} we define the GIG and Wishart distributions and state the Matsumoto-yYor property on symmetric cones. 
A core of the proof of the converse to the Matsumoto-yYor property is a solution of some functional equation for real functions with arguments from the cone. Section \ref{funeq} is devoted to analysis of this functional equation. The statement and the proof of the main result are given in Sect. \ref{secMY}. Finally, in Sect. \ref{ddd} we give some remarks regarding the MY property on matrices of different dimensions and related functional equation.

\section{Symmetric Cones}
In this section, we give a short introduction to the theory of symmetric cones. For further details, we refer to \citep{FaKo1994}. 

A \textit{Euclidean Jordan algebra} is a Euclidean space $\En$ (endowed with the scalar product denoted by $\scalar{\xx,\yy}$) equipped with a bilinear mapping (product)
\begin{align*}
\En\times\En \ni \left(\xx,\yy\right)\mapsto \xx\yy\in\En
\end{align*}
and a neutral element $\ee$ in $\En$ such that for all $\xx$, $\yy$, $\zz$ in $\En$:
\begin{itemize}
	\item $\xx\yy=\yy\xx$, 
	\item $\xx(\xx^2\yy)=\xx^2(\xx\yy)$,
	\item $\xx\ee=\xx$,
	\item $\scalar{\xx,\yy\zz}=\scalar{\xx\yy,\zz}$.
\end{itemize}
For $\xx\in\En$ let $\LL(\xx)\colon \En\to\En$ be the linear map defined by
\begin{align*}
\LL(\xx)\yy=\xx\yy,
\end{align*}
and define 
\begin{align*}
\PP(\xx)=2\LL^2(\xx)-\LL\left(\xx^2\right).
\end{align*} 
Let $\mathrm{End}(\En)$ denote the space of endomorphisms of $\En$. The map $\PP\colon \En\mapsto \mathrm{End}(\En)$ is called the \emph{quadratic representation} of $\En$.

An element $\xx$ is said to be \emph{invertible} if there exists an element $\yy$ in $\En$ such that $\LL(\xx)\yy=\ee$. Then, $\yy$ is called the \emph{inverse of} $\xx$ and it is denoted by $\yy=\xx^{-1}$. Note that the inverse of $\xx$ is unique. It can be shown that $\xx$ is invertible if and only if $\PP(\xx)$ is invertible, and in this case, $\left(\PP(\xx)\right)^{-1} =\PP\left(\xx^{-1}\right)$.

A Euclidean Jordan algebra $\En$ is said to be \emph{simple} if it is not a \mbox{Cartesian} product of two Euclidean Jordan algebras of positive dimensions. Up to linear isomorphism, there are only five kinds of Euclidean simple Jordan algebras. Let $\mathbb{K}$ denote either the real numbers $\RR$, the complex ones $\mathbb{C}$, the quaternions $\mathbb{H}$ or the octonions $\mathbb{O}$. Let us write $S_r(\mathbb{K})$ for the space of $r\times r$ Hermitian matrices valued in $\mathbb{K}$, endowed with the Euclidean structure $\scalar{\xx,\yy}=\Trace(\xx\cdot\bar{\yy})$ and with the Jordan product
\begin{align*}
\xx\yy=\tfrac{1}{2}(\xx\cdot\yy+\yy\cdot\xx),
\end{align*}
where $\xx\cdot\yy$ denotes the ordinary product of matrices and $\bar{\yy}$ is the conjugate of $\yy$. Then $S_r(\RR)$, $r\geq 1$, $S_r(\mathbb{C})$, $r\geq 2$, $S_r(\mathbb{H})$, $r\geq 2$, and the exceptional $S_3(\mathbb{O})$ are the first four kinds of Euclidean simple Jordan algebras. Note that in this case if $\KK\neq \mathbb{O}$, then
\begin{align*}
\PP(\yy)\xx=\yy\cdot\xx\cdot\yy.
\end{align*}
The fifth kind is the Euclidean space $\RR^{n+1}$, $n\geq 2$, with the Jordan product
\begin{align}\label{scL}\begin{split}
\left(x_0,x_1,\dots, x_n\right)\left(y_0,y_1,\ldots,y_n\right) =\left(\sum_{i=0}^n x_i y_i,x_0y_1+y_0x_1,\ldots,x_0y_n+y_0x_n\right).
\end{split}
\end{align}
To each Euclidean simple Jordan algebra, one can attach the set $\bar{\VV}$ of Jordan squares
\begin{align*}
\bar{\VV}=\left\{\xx\in\En\colon\mbox{ there exists }\yy\mbox{ in }\En\mbox{ such that }\xx=\yy^2 \right\}.
\end{align*}
The interior $\VV$ is a \emph{symmetric cone}.
Moreover, $\VV$ is \emph{irreducible}, i.e., it is not the Cartesian product of two convex cones. One can prove that an open convex cone is symmetric and irreducible if and only if it is the symmetric cone $\VV$ of some Euclidean simple Jordan algebra. Each simple Jordan algebra corresponds to a symmetric cone; hence, there exists up to linear isomorphism also only five kinds of symmetric cones. The cone corresponding to the Euclidean Jordan algebra $\RR^{n+1}$ equipped with Jordan product \eqref{scL} is called the Lorentz cone. 

We will now introduce a very useful decomposition in $\En$, called the \emph{spectral decomposition}. An element $\cb\in\En$ is said to be a \emph{primitive idempotent} if $\cb\cb=\cb\neq 0$ and if $\cb$ is not a sum of two non-null idempotents. A \emph{complete system of primitive orthogonal idempotents} is a set $\left(\cb_1,\ldots,\cb_r\right)$ such that
\begin{align*}
\sum_{i=1}^r \cb_i=\ee\quad\mbox{and}\quad\cb_i\cb_j=\delta_{ij}\cb_i\quad\mbox{for } 1\leq i\leq j\leq r.
\end{align*}
The size $r$ of such system is a constant called the \emph{rank} of $\En$. Any element $\xx$ of a Euclidean simple Jordan algebra can be written as $\xx=\sum_{i=1}^r\lambda_i\cb_i$ for some complete system of primitive orthogonal idempotents $\left(\cb_1,\ldots,\cb_r\right)$. The real numbers $\lambda_i$, $i=1,\dots,r$ are the \emph{eigenvalues} of $\xx$. One can then define the \emph{trace} and the \emph{determinant} of $\xx$ by, respectively, $\tr\xx=\sum_{i=1}^r\lambda_i$ and $\det\xx=\prod_{i=1}^r\lambda_i$. An element $\xx\in\En$ belongs to $\VV$ if and only if all its eigenvalues are strictly positive. 

Note that up to a multiplicative constant, $\tr(\xx\yy)$ is the only scalar product of $\En$ which makes $\Omega$ self dual. Henceforth we assume that $\VV$ is an irreducible cone and that corresponding Jordan algebra $\En$ is equipped with canonical scalar product $\scalar{\xx,\yy}=\tr(\xx\yy)$.

The rank $r$ and $\dim\VV$ of irreducible symmetric cone are connected through the relation
\begin{align*}
\dim\VV=r+\frac{d r(r-1)}{2},
\end{align*}
where $d$ is an integer called the \emph{Peirce constant}.

The important property of the determinant is that
\begin{align}
\det\left(\PP(\xx)\yy\right)=(\det\xx)^2 \det\yy,\quad (\xx,\yy)\in\VV^2. \label{wCwC}
\end{align}
It turns out that \eqref{wCwC} characterizes determinant -- see Lemma \ref{Pexider} below. Moreover (see \cite[Proposition II.4.2]{FaKo1994})
\begin{align}\label{dddet}
\DDet\left(\PP(\xx)\right)=(\det\xx)^{2\dim\VV/r},
\end{align}
where $\DDet$ denotes the determinant in the space of endomorphisms on $\VV$.

In the proof of our main theorem we will need the following identity (called Hua's identity - see \cite[Exercise 5c, p.39]{FaKo1994})
\begin{align}\label{Hua}
\ab^{-1}-(\ab+\bb)^{-1}=(\ab+\PP(\ab)\bb^{-1})^{-1}
\end{align}
when $\ab\in\VV$, $\bb\in\En$ are such that $\bb$, $\ab+\bb$ and $\ab+\PP(\ab)\bb^{-1}$ are invertible. Note that if $\ab, \bb\in\VV$, then $\ab^{-1}-(\ab+\bb)^{-1}\in\VV$. For the cone $\Omega_+$ of symmetric positive definite real matrices, Hua's identity takes the form given in \eqref{huavp}.

\section{Wishart and GIG Distributions}\label{secProb}
The Wishart distribution $\gamma_{p,\ab}$ in $\bar{\VV}$ is defined for any $\ab\in\VV$ and any $p$ in the set
\begin{align*}
\Lambda=\{0,d/2,d,\ldots,d(r-1)/2\}\cup(d(r-1)/2,\infty)
\end{align*}
by its Laplace transform
\begin{align*}
\int_{\bar{\VV}} \exp(-\scalar{\sigma,\yy})\gamma_{p,\ab}(\mathrm{d}\yy)=\left(\frac{\det\ab}{\det\left(\ab+\sigma\right)}\right)^p,
\end{align*}
which holds for any $\sigma+\ab\in\VV$. If $p>\dim\VV/r-1$, then $\gamma_{p,\ab}$ is absolutely continuous with respect to the Lebesgue measure and has the density
\begin{align*}
\gamma_{p,\ab}(\mathrm{d}\xx)=\frac{(\det\ab)^p}{\Gamma_\VV(p)} (\det\xx)^{p-\dim\VV/r}\mathrm{e}^{-\scalar{\ab,\xx}}I_\VV(\xx)\,\mathrm{d}\xx,\quad\xx\in\VV,
\end{align*}
where $\Gamma_\VV$ is the gamma function of the symmetric cone $\VV$ (see \cite[p.124]{FaKo1994}).

The absolutely continuous generalized inverse Gaussian distribution $\mu_{p,\ab,\bb}$ on $\VV$ is defined for $\ab,\bb\in\VV$ and $p\in\RR$ by its density
\begin{align*}
\mu_{p,a,b}(\mathrm{d}\xx)=\frac{1}{K_p(\ab,\bb)} (\det\xx)^{p-\dim\VV/r}e^{-\scalar{\ab,\xx}-\scalar{\bb,\xx^{-1}}}I_\VV(\xx)\,\mathrm{d}\xx,\quad\xx\in\VV,
\end{align*}
where $K_p(\ab,\bb)$ is a normalizing constant.

In \citep{WesLe00}, Theorem \ref{TXXX} was proved in the special case of the cone of symmetric positive definite real matrices $\VV_+$. As it was observed by the authors, symmetric cones are the natural framework for considering the Matsumoto--Yor property. We state the following theorem without a proof as it only mimics the argument for $\VV_+$. The original proof relies on the properties of Bessel-like functions ($K_p(\ab,\bb)$) introduced in \citep{Herz55}, which retain their usual properties in the symmetric cone setting.
\begin{theorem}\label{TXXX}
Let $p\in\Lambda$ and $\ab$ and $\bb$ in irreducible symmetric cone $\VV$. Let $X$ and $Y$ be independent random variables in $\VV$ and $\bar{\VV}$ with respective distributions $\mu_{-p,\ab,\bb}$ and $\gamma_{p,\ab}$. Then random variables $U=(X+Y)^{-1}$ and $V=X^{-1}-(X+Y)^{-1}$ are independent with respective distributions $\mu_{-p,\bb,\ab}$ and $\gamma_{p,\bb}$.
\end{theorem}
 
\section{Functional Equations}\label{funeq}
At the beginning of this section we state three results that will be useful in the proof of the main technical result - Theorem \ref{th1}. The first one regards regular additive functions (see \citep{Kucz09}) on symmetric cone. 

\begin{lemma}[Additive Cauchy functional equation]\label{Additive}
Let $f\colon \VV\to\RR$ be a measurable function such that
\begin{align*}
f(\xx)+f(\yy)=f(\xx+\yy),\quad (\xx,\yy)\in\VV^2.
\end{align*}
Then there exists $\ff\in\En$ such that $f(\xx)=\scalar{\ff,\xx}$ for any $\xx\in\VV$.
\end{lemma}
An elementary proof of this theorem may be found in \citep{Kolo2013}.
The following lemma was recently proved in \citep{wC2013}.
\begin{lemma}[Logarithmic Pexider functional equation]\label{Pexider}
Let $f_1$, $f_2$, $f_3\colon \VV\to\RR$ be measurable functions such that
\begin{align*}
f_1(\xx)+f_2(\yy)=f_3\left(\PP\left(\xx^{1/2}\right)\yy\right),\quad (\xx,\yy)\in\VV^2.
\end{align*}
Then there exist a constant $q\in\RR$ and constants $\gamma_1$, $\gamma_2\in\RR$ such that for all $\xx\in\VV$,
\begin{align*}
f_1(\xx) & =q\log\det\xx+\gamma_1, \\
f_2(\xx) & =q\log\det\xx+\gamma_2, \\
f_3(\xx) & =q\log\det\xx+\gamma_1+\gamma_2.
\end{align*}
\end{lemma}

The main technical result will rely on the following univariate result due to \citet{Wes022}.
\begin{theorem}\label{1dimwes}
Let $A$, $B$, $C$ and $D$ be locally integrable real functions defined on $(0,\infty)$ such that
\begin{align}\label{eq10}
g(x(x+y))-g(y(x+y))=\alpha(x)-\alpha(y),\quad (x,y)\in(0,\infty)^2.\quad
\end{align}
Then there exist real numbers $A$, $B$, $C$ and $D$ such that for any $x>0$, 
\begin{align*}
g(x)=Ax+B\log x+C,\quad \alpha(x)=Ax^2+B\log x+D.
\end{align*}
\end{theorem}
The following result then follows from Theorem \ref{1dimwes}.
\begin{theorem}\label{1dim}
Let $A$, $B$, $C$ and $D$ be locally integrable real functions defined on $(0,\infty)$ such that
\begin{align}\label{eq1}
A(x)+B(y)=C\left((x+y)^{-1}\right)+D\left(x^{-1}-(x+y)^{-1}\right),\quad (x,y)\in(0,\infty)^2.
\end{align}
Then there exist real numbers $p$, $f$, $g$ and $C_i$, $i=1,\ldots,4$, such that for any $x>0$, 
\begin{align*}
A(x) & =-p\log x+fx+gx^{-1}+C_1,\\
B(x) & =p\log x+fx+C_2,\\
C(x) & =-p\log x+g x+f x^{-1}+C_3,\\
D(x) & =p\log x+g x+C_4,
\end{align*}
and $C_1+C_2=C_3+C_4$.
\end{theorem}
\begin{proof}
Denote $g_1(x)=A(x^{-1})-B(x^{-1})$ and $\alpha_1(x)=D(x^2)$. Interchange the roles of $x$ and $y$ in \eqref{eq1} and subtract from the original equation. Then
\begin{align*}
g_1\left(x^{-1}\right)-g_1\left(y^{-1}\right)=\alpha_1\left(\sqrt{\frac{y}{x(x+y)}}\right)-\alpha_1\left(\sqrt{\frac{x}{y(x+y)}}\right).
\end{align*}
Inserting $x=(u(u+v))^{-1}$ and $y=(v(u+v))^{-1}$, we arrive at \eqref{eq10} with $g$ and $\alpha$ replaced, respectively, with $g_1$ and $\alpha_1$.

Substituting $x\mapsto (x+y)^{-1}$ and $y\mapsto x^{-1}-(x+y)^{-1}$ in \eqref{eq1}, we obtain 
\begin{align*}
A\left((x+y)^{-1}\right)+B\left(x^{-1}-(x+y)^{-1}\right)=C(x)+D(y),\quad (x,y)\in(0,\infty)^2.
\end{align*}
As before, denoting $g_2(x)=C(x^{-1})-D(x^{-1})$ and $\alpha_2(x)=B(x^2)$ and subtracting the same equation with $x$ and $y$ interchanged, we see that \eqref{eq10} holds true for $g_2$ and $\alpha_2$ also. Functions $g_i$ and $\alpha_i$, $i=1,2$, are locally integrable, because for $g_1$ we have
$$\int_K|A(x^{-1})-B(x^{-1})|\,\mathrm{d}x=\int_{\phi(K)}|A(y)-B(y)|\frac{\mathrm{d}y}{y^2}\leq c\int_{\phi(K)}|A(y)-B(y)|\,\mathrm{d}y$$
for all compact sets $K\subset (0,\infty)$, where $\phi(K)$ is the (compact) image of $K$ under $\phi(x)=x^{-1}$. Since $A$ and $B$ were assumed to be locally integrable, we see that $g_1$ is locally integrable. Analogously, we proceed for $g_2$, $\alpha_1$ and $\alpha_2$.
Thus, by Theorem \ref{1dimwes}, we obtain (we borrow this notation from Theorem \ref{1dimwes}):
\begin{align*}
B(x) & =\alpha_2(\sqrt{x})=A_2x+B_2/2\log x+D_2, \\
D(x) & =\alpha_1(\sqrt{x})=A_1x+B_1/2\log x+D_1, \\
A(x) & = A(x)-B(x)+B(x)=g_1(x^{-1})+\alpha_2(\sqrt{x})=A_2x+A_1x^{-1}-(B_1-B_2/2)\log x+C_1+D_2, \\
C(x) & =C(x)-D(x)+D(x)=g_2(x^{-1})+\alpha_1(\sqrt{x})=A_1x+A_2x^{-1}-(B_2-B_1/2)\log x+C_2+D_1, 
\end{align*}
Inserting it back into \eqref{eq1}, it can be quickly verified that $B_1=B_2=B$. 
\end{proof}
We are now ready to state and solve the functional equation related to the Matsumoto--Yor property on symmetric cones. 

\begin{theorem}\label{th1}
Let $a$, $b$, $c$ and $d$ be continuous real functions defined on $\VV$ such that
\begin{align}\label{eq2}
a(\xx)+b(\yy)=c\left((\xx+\yy)^{-1}\right)+d\left(\xx^{-1}-(\xx+\yy)^{-1}\right),\quad (\xx,\yy)\in\VV^2.
\end{align}
Then there exist constants $q\in\RR$, $\ff$, $\gb\in\En$ and $\gamma_i\in\RR$, $i=1, 2, 3$, such that for any $\xx\in\VV$, 
\begin{align*}
a(\xx) & =q\log \det\xx+\scalar{\ff,\xx}+\scalar{\gb,\xx^{-1}}+\gamma_1+\gamma_3,\\
b(\xx) & =-q\log \det\xx+\scalar{\ff,\xx}+\gamma_2,\\
c(\xx) & =q\log \det\xx+\scalar{\gb,\xx}+\scalar{\ff,\xx^{-1}}+\gamma_3,\\
d(\xx) & =-q\log \det\xx+\scalar{\gb,\xx}+\gamma_1+\gamma_2.
\end{align*}
\end{theorem}
\begin{proof}
By inserting $(\xx,\yy)=(\alpha\zz,\beta\zz)$ for $\alpha, \beta>0$ and $\zz\in\VV$ into \eqref{eq2}, we arrive at the equation \eqref{eq1}
with $A(\alpha):=a(\alpha\zz)$, $B(\alpha):=b(\alpha\zz)$, $C(\alpha):=c(\alpha\zz^{-1})$ and $D(\alpha):=d(\alpha\zz^{-1})$.
Functions $A$, $B$, $C$ and $D$ are continuous, so they are locally integrable. Therefore, by Theorem \ref{1dim}, for any $\zz\in\VV$, there exist constants $p(\zz)$, $f(\zz)$, $g(\zz)$ and $C_i(\zz)$, $i=1,\ldots,4$, such that
\begin{align}\label{abcd}\begin{split}
a(\alpha\zz) & =-p(\zz)\log \alpha+f(\zz)\alpha+g(\zz)\alpha^{-1}+C_1(\zz),\\
b(\alpha\zz) & =p(\zz)\log \alpha+f(\zz)\alpha+C_2(\zz),\\
c(\alpha\zz^{-1}) & =-p(\zz)\log \alpha+g(\zz) \alpha+f(\zz) \alpha^{-1}+C_3(\zz),\\
d(\alpha\zz^{-1}) & =p(\zz)\log \alpha+g(\zz)\alpha+C_4(\zz),\\
C_1(\zz)+C_2(\zz)&=C_3(\zz)+C_4(\zz),
\end{split}\end{align}
for any $\alpha>0$ and $\zz\in\VV$. Functions $\zz\mapsto p(\zz)$, $\zz\mapsto f(\zz)$, $\zz\mapsto g(\zz)$ and $\zz\mapsto C_i(\zz)$, $i=1,\ldots,4$, are continuous, because $a$, $b$, $c$ and $d$ are continuous. Let $\beta>0$. By the equality $a(\alpha(\beta\zz))=a((\alpha\beta)\zz)$, we obtain that for any $\alpha>0$,
\begin{align*}
a(\alpha\beta\zz)=&-p(\zz)\log \alpha\beta+f(\zz)\alpha\beta+g(\zz)\alpha^{-1}\beta^{-1}+C_1(\zz) \\
=&-p(\beta\zz)\log \alpha+f(\beta\zz)\alpha+g(\beta\zz)\alpha^{-1}+C_1(\beta\zz),
\end{align*}
hence
\begin{align}\label{fgpC}\begin{split}
f(\beta\zz)&=\beta f(\zz),\qquad g(\beta\zz)=\beta^{-1}g(\zz), \\
p(\beta\zz)&=p(\zz), \quad C_1(\beta\zz)=C_1(\zz)-p(\zz)\log\beta.
\end{split}\end{align}
Following the same procedure for functions $b$, $c$ and $d$, we have
\begin{align}\label{Ci}\begin{split}
C_i(\beta\zz)&=C_i(\zz)+p(\zz)\log\beta,\quad i=2,3,\\
C_4(\beta\zz)&=C_4(\zz)-p(\zz)\log\beta.
\end{split}\end{align}
Using \eqref{abcd} for $\alpha=1$ in \eqref{eq2}, we get
\begin{multline}\label{eq4}
f(\xx)+g(\xx)+C_1(\xx)+f(\yy)+C_2(\yy)\\
=g(\xx+\yy)+f(\xx+\yy)+C_3(\xx+\yy)\\
+\,g\left((\xx^{-1}-(\xx+\yy)^{-1})^{-1}\right)+C_4\left((\xx^{-1}-(\xx+\yy)^{-1})^{-1}\right).
\end{multline}
Consider the above equation for $(\alpha^{-1}\xx,\alpha^{-1}\yy)\in\VV^2$, $\alpha>0$. Then, by \eqref{fgpC},
\begin{multline*}
\alpha^{-1} f(\xx)+\alpha g(\xx)+C_1(\alpha^{-1}\xx)+\alpha^{-1} f(\yy)+C_2(\alpha^{-1}\yy) \\
\quad=\alpha g(\xx+\yy)+\alpha^{-1} f(\xx+\yy)+C_3(\alpha^{-1}(\xx+\yy))\\
\quad\quad+\,\alpha g\left((\xx^{-1}-(\xx+\yy)^{-1})^{-1}\right)+C_4\left(\alpha^{-1}(\xx^{-1}-(\xx+\yy)^{-1})^{-1}\right).
\end{multline*}
Multiplying both sides of the above equation by $\alpha$ and passing to the limit as $\alpha\to0$, by \eqref{Ci}, we obtain
\begin{multline*}
f(\xx)+ f(\yy)-f(\xx+\yy)=\\
\quad=\lim_{\alpha\to0}\alpha\left\{C_3(\alpha^{-1}(\xx+\yy))+C_4\left(\alpha^{-1}(\xx^{-1}-(\xx+\yy)^{-1})^{-1}\right)\right.\\
\qquad \quad \quad \quad \quad \left.-C_1(\alpha^{-1}\xx)-C_2(\alpha^{-1}\yy)\right\}.
\end{multline*}
By \eqref{fgpC} and \eqref{Ci}, the limit on the right-hand side of the above equation equals $0$. Thus, by Lemma \ref{Additive}, there exists $\ff\in\En$ such that $f(\xx)=\scalar{\ff,\xx}$. Analogously, consider \eqref{eq4} for $(\alpha\xx,\alpha\yy)\in\VV^2$, $\alpha>0$, multiply its both sides by $\alpha$ and pass to the limit as $\alpha\to0$.
Then
\begin{multline*}
g(\xx)-g(\xx+\yy)-g\left((\xx^{-1}-(\xx+\yy)^{-1})^{-1}\right) \\
=\lim_{\alpha\to0} \alpha \left\{ C_3(\alpha(\xx+\yy))+ C_4\left(\alpha(\xx^{-1}-(\xx+\yy)^{-1})^{-1}\right)\right.\\
\qquad \quad \quad \quad \quad \left.-C_1(\alpha\xx)-C_2(\alpha\yy)\right\}=0.
\end{multline*}
Define $\bar{g}(\xx)=g(\xx^{-1})$. Then,
\begin{align*}
\bar{g}(\xx^{-1})=\bar{g}((\xx+\yy)^{-1})+\bar{g}(\xx^{-1}-(\xx+\yy)^{-1}).
\end{align*}
Thus, $\bar{g}$ is additive, i.e., there exists $\gb\in\En$ such that $g(\xx)=\scalar{\gb,\xx^{-1}}$.

By the use of above results for $f$ and $g$, \eqref{eq4} simplifies to
\begin{align}\label{newMain}
C_1(\xx)+C_2(\yy)=C_3(\xx+\yy)+C_4\left((\xx^{-1}-(\xx+\yy)^{-1})^{-1}\right).
\end{align}
Recall that by Hua's identity \eqref{Hua}, the argument of $C_4$ above may be written as
$$(\xx^{-1}-(\xx+\yy)^{-1})^{-1}=\xx+\PP(\xx)\yy^{-1}.$$
Using this fact along with \eqref{Ci} in \eqref{newMain} for $\yy=\alpha\zz$, we obtain
\begin{multline*}
C_1(\xx)+C_2(\zz)+p(\zz)\log\alpha\\
\quad =C_1(\xx)+C_2(\alpha\zz)=C_3(\xx+\alpha\zz)+C_4\left(\alpha^{-1}(\alpha\xx+\PP(\xx)\zz^{-1})\right)\\
\quad =C_3(\xx+\alpha\zz)+C_4(\alpha\xx+\PP(\xx)\zz^{-1})+p(\alpha\xx+\PP(\xx)\zz^{-1})\log\alpha.
\end{multline*}
Passing to the limit as $\alpha\to0$ (recall that $C_i$ are continuous on $\VV$), we obtain
\begin{align}\label{limCi}
C_1(\xx)+C_2(\zz)-C_3(\xx)-C_4(\PP(\xx)\zz^{-1})=\lim_{\alpha\to0}\log\alpha\left\{p\left(\alpha\xx+\PP(\xx)\zz^{-1}\right)-p(\zz)\right\}
\end{align}
for any $(\xx,\zz)\in\VV^2$. A necessary condition for the limit on the right-hand side to exist is 
$$\lim_{\alpha\to0}\left\{p(\alpha\xx+\PP(\xx)\zz^{-1})-p(\zz)\right\}=0.$$
But $p$ is continuous and $\lim_{\alpha\to0}p(\alpha\xx+\PP(\xx)\zz^{-1})=p(\PP(\xx)\zz^{-1})$, hence $p(\zz)=p(\PP(\xx)\zz^{-1})$. Thus, function $p$ is constant and the right-hand side of \eqref{limCi} is equal to $0$.
Hence, substituting $\zz=\yy^{-1}$ and $\xx\mapsto\xx^{1/2}$ in \eqref{limCi}, we get
$$C_1(\xx^{1/2})-C_3(\xx^{1/2})+C_2(\yy^{-1})=C_4(\PP(\xx^{1/2})\yy).$$
Define $f_1(\xx):=C_1(\xx^{1/2})-C_3(\xx^{1/2})$, $f_2(\xx):=C_2(\xx^{-1})$ and $f_3(\xx):= C_4(\xx)$ for $\xx\in\VV$. 
Then 
$$f_1(\xx)+f_2(\yy)=f_3(\PP(\xx^{1/2})\yy),\quad (\xx,\yy)\in\VV^2.$$
By Lemma \ref{Pexider}, there exist real constants $q$, $\gamma_1$ and $\gamma_2$ such that for any $\xx\in\VV$,
\begin{align*}
f_1(\xx)&=q\log\det\xx+\gamma_1,\\
f_2(\xx)&=q\log\det\xx+\gamma_2,\\
f_3(\xx)&=q\log\det\xx+\gamma_1+\gamma_2,
\end{align*}
that is,
\begin{align*}
C_1(\xx)&=C_3(\xx)+2q\log\det\xx+\gamma_1,\\
C_2(\xx)&=-q\log\det\xx+\gamma_2,\\
C_4(\xx)&=q\log\det\xx+\gamma_1+\gamma_2.
\end{align*}

Let us go back to \eqref{newMain} and use the above result. Then
\begin{align*}
C_3(\xx)+2q\log\det\xx-q\log\det\yy =C_3(\xx+\yy)+q\log\det(\xx+\PP(\xx)\yy^{-1}),\quad (\xx,\yy)\in\VV^2.
\end{align*}
Since $\det(\xx+\PP(\xx)\yy^{-1})=\det(\xx^2)\det(\xx^{-1}+\yy^{-1})$, we obtain
\begin{align*}
C_3(\xx)-q\log\det\yy=C_3(\xx+\yy)+q\log\det\left(\xx^{-1}+\yy^{-1}\right).
\end{align*}
One can interchange $\xx$ and $\yy$ on the right-hand side to obtain
\begin{align*}
C_3(\xx)+q\log\det\xx=C_3(\yy)+q\log\det\yy=\mathrm{const}:=\gamma_3,
\end{align*}
that is, $C_3(\xx)=-q\log\det\xx+\gamma_3$, what completes the proof. 
\end{proof}

\section{Main Result}\label{secMY}
In the following section, we prove our main result, which is a converse to the Matsumoto--Yor property in the symmetric cone-variate case. We reduce the smoothness conditions for densities from $C^2$ densities in \citep{WesLe00} and differentiability in \citep{Wes02} to the continuity only. 

\begin{theorem}\label{thP}
Let $X$ and $Y$ be independent random variables in $\VV$ with continuous and strictly positive densities. If the random variables $U=(X+Y)^{-1}$ and $V=X^{-1}-(X+Y)^{-1}$ are independent, then there exists $p>\dim\VV/r-1$, $\ab$ and $\bb$ in $\VV$ such that $X$ and $Y$ follow respective distributions $\mu_{-p,\ab,\bb}$ and $\gamma_{p,\ab}$.
\end{theorem}
\begin{proof}
Define the map $\Psi\colon\VV^2\to\VV^2$ by $\Psi(\xx,\yy)=\left((\xx+\yy)^{-1},\xx^{-1}-(\xx+\yy)^{-1}\right)=(\ub,\vb)$. Obviously, $(U,V)=\Psi(X,Y)$.
Function $\Psi$ is a bijection. In order to find the joint density of $(U,V)$ the essential computation is the one involved with finding the Jacobian $J$ of the map $\psi^{-1}$, that is, the determinant of the linear map
\begin{align*}
\begin{pmatrix}
\mathrm{d}\ub\\
\mathrm{d}\vb
\end{pmatrix}
\mapsto
\begin{pmatrix}
\mathrm{d}\xx \\
\mathrm{d}\yy 
\end{pmatrix}
=
\begin{pmatrix}
\mathrm{d}\xx/\mathrm{d}\ub & \mathrm{d}\xx/\mathrm{d}\vb \\
\mathrm{d}\yy/\mathrm{d}\ub & \mathrm{d}\yy/\mathrm{d}\vb
\end{pmatrix}
\begin{pmatrix}
\mathrm{d}\ub\\
\mathrm{d}\vb
\end{pmatrix}.
\end{align*}
It is easy to see that $\Psi=\Psi^{-1}$, that is $(\xx,\yy)=\left((\ub+\vb)^{-1},\ub^{-1}-(\ub+\vb)^{-1}\right)$. Note that the derivative of the map $\xx\mapsto\xx^{-1}$ is $-\PP(\xx)^{-1}$. Thus
\begin{align*}
J & =\left| 
\begin{array}{cc}
-\PP(\ub+\vb)^{-1} & -\PP(\ub+\vb)^{-1} \\
-\PP(\ub)^{-1}+\PP(\ub+\vb)^{-1} & \PP(\ub+\vb)^{-1}
\end{array}
\right| 
=\left| 
\begin{array}{cc}
-\PP(\ub)^{-1} & 0 \\
-\PP(\ub)^{-1}+\PP(\ub+\vb)^{-1} & \PP(\ub+\vb)^{-1}
\end{array}
\right|\\
& =\DDet\left(\PP(\ub+\vb)^{-1}\PP(\ub)^{-1}\right).
\end{align*}
By \eqref{dddet}, we get
\begin{align*}
J=(\det\ub\det(\ub+\vb))^{-2\dim\VV/r}.
\end{align*}
Since $(X,Y)$ and $(U,V)$ have independent components, the following identity holds almost everywhere with respect to the Lebesgue measure: 
\begin{align*}
f_U(\ub)f_V(\vb)=(\det\ub\det(\ub+\vb))^{-2\dim\VV/r} f_X\left((\ub+\vb)^{-1}\right)f_Y\left(\ub^{-1}-(\ub+\vb)^{-1}\right),
\end{align*}
where $f_X$, $f_Y$, $f_U$ and $f_V$ denote densities of $X$, $Y$, $U$ and $V$, respectively.
Since the respective densities are assumed to be continuous, the above equation holds for every $(\ub,\vb)\in\VV^2$. Taking the logarithms of both sides of the above equation (it is permitted since $f_X$, $f_Y>0$ on $\VV$), we get 
\begin{align}\label{lukacs}
a(\ub)+b(\vb)=c\left((\ub+\vb)^{-1}\right)+d\left(\ub^{-1}-(\ub+\vb)^{-1}\right),
\end{align}
where
\begin{align*}
a(\xx)&=\log\, f_U(\xx)+\tfrac{2\dim\VV}{r}\log\det\xx,\\
c(\xx)&=\log\, f_X(\xx)+\tfrac{2\dim\VV}{r}\log\det\xx,\\
b&=\log\, f_V,\qquad d=\log\, f_Y.
\end{align*}
By Theorem \ref{th1}, there exist constants $q\in\RR$, $\ff$, $\gb\in\En$ and $\gamma_i\in\RR$, $i=1, 2, 3$, such that for any $\xx\in\VV$, 
\begin{align*}
c(\xx) & =-q\log \det\xx+\scalar{\gb,\xx}+\scalar{\ff,\xx^{-1}}+\gamma_3,\\
d(\xx) & =q\log \det\xx+\scalar{\gb,\xx}+\gamma_1+\gamma_2,
\end{align*}
that is,
\begin{align*}
f_X(\xx) & = \mathrm{e}^{\gamma_3} (\det\xx)^{-q-2\dim\VV/r} \mathrm{e}^{\scalar{\gb,\xx}+\scalar{\ff,\xx^{-1}}},\\
f_Y(\xx) & = \mathrm{e}^{\gamma_1+\gamma_2}(\det\xx)^q \mathrm{e}^{\scalar{\gb,\xx}}.
\end{align*}
Since $f_X$ and $f_Y$ are some densities, we have $\ab=-\gb\in\VV$, $\bb=-\ff\in\VV$ and $q=p-\dim\VV/r>-1$. Thus, $X\sim\mu_{-p,\ab,\bb}$ and $Y\sim\gamma_{p,\ab}$. 
\end{proof}

\section{Comments}\label{ddd}
Recall that $S_r(\KK)$ denotes the space of $r\times r$ Hermitian matrices valued in $\KK$. Let $\VV_r(\KK)$ be the symmetric cone of Jordan algebra  $\En=S_r(\KK)$, where $\KK$ denotes either the real numbers $\RR$, the complex ones $\mathbb{C}$ or the quaternions $\mathbb{H}$. We exclude here the non-associative case $\KK=\mathbb{O}$.

Let $z$ be a fixed $s\times r$ matrix of full rank valued in $\KK$ and define the linear mapping $\PP_{sr}\colon S_r(\KK)\to S_s(\KK)$ by
$$\PP_{sr}(z)\xx=z\cdot\xx\cdot z^\ast.$$
If $r=s$, then $\PP_{sr}$ is the ordinary quadratic representation of $\Omega_s$. In the rest of the paper, we will drop the subscript and simply write $\PP$ (abusing the notation from previous sections).

Now, consider the following transformation $\psi_z\colon \VV_r(\KK)\times\VV_s(\KK)\to\VV_s(\KK)\times\VV_r(\KK)$, where
$$\psi_z(\xx,\yy)=\left((\PP(z)\xx+\yy)^{-1},\xx^{-1}-\PP(z^\ast)(\PP(z)\xx+\yy)^{-1}\right).$$
It is natural to ask whether an analogue of Theorem \ref{thP} holds if we consider independent random variables $X$ and $Y$ valued in $\VV_r(\KK)$ and $\VV_s(\KK)$ and define $(U,V)=\psi_z(X,Y)$. The answer is affirmative, and it was given in \cite[Theorem 4.1]{MasWes06}. Following the same steps as in the proof of Theorem \ref{thP}, the problem of characterization of probability measures is reduced to the problem of solving following functional equation
	\begin{align}\label{maindiff}
	a(\ub)+b(\vb)=c\left((\PP(z^\ast)\ub+\vb)^{-1}\right)+d\left(\ub^{-1}-\PP(z)(\PP(z^\ast)\ub+\vb)^{-1}\right),\quad (\ub,\vb)\in\VV_s(\KK)\times\VV_r(\KK),
	\end{align}
where $a,d\colon \VV_s(\KK)\to\RR$ and $b,c\colon \VV_r(\KK)\to\RR$ are some unknown functions.
This functional equation was solved by \citet{MasWes06} for $\KK=\RR$ under the assumption that the unknown functions are differentiable. It can be shown that through Theorem \ref{th1}, this assumption may be weakened to continuity. Therefore, we obtain the following refinement of \cite[Theorem 4.1]{MasWes06}:
\begin{theorem}
	Let $X$ and $Y$ be independent random variables with values in $\VV_r(\KK)$ and $\VV_s(\KK)$, respectively. Assume that $X$ and $Y$ have continuous densities, which are strictly positive.
	Define $(U,V)=\psi_z(X,Y)$. 
	
	If $U$ and $V$ are independent, then there exist matrices $(\ab,\bb)\in \VV_s(\KK)\times\VV_r(\KK)$ and a constant $p>\dim\VV_{r}(\KK)/r-1$ such that  $$(X,Y)\sim\mu^{(r)}_{-p,\PP(z^\ast)\ab,\bb}\otimes\gamma^{(s)}_{q,\ab},$$ where $q=p+(\dim\VV_{s}(\KK)/s-\dim\VV_{r}(\KK)/r)$.
\end{theorem}
The superscripts $^{(s)}$ and $^{(r)}$ are used to emphasize the ranks of the cones on which the distributions are considered.
 
	The solution to \eqref{maindiff} was also used in the proof of the characterization of Wishart distribution through its block conditional independence structure (see \citep[Theorem 5.1]{MasWes06}. One of the technical assumptions was that the respective random matrix has a differentiable density. This was assumed only in order to solve a functional equation, whose solution was not known under weaker assumptions. Therefore, this assumption may be reduced to the existence of continuous densities. 
		
An analogous assumption was imposed on the densities in the recent paper of \citet{Bob15}, where the multivariate MY property on trees is considered -- see \citep[Theorem 4.3]{Bob15}. Thanks to the solution of \eqref{maindiff} under weaker assumptions, this theorem holds true if we assume continuity of densities only.
\subsection*{Acknowledgment} This research was partially supported by NCN Grant No. 2012/05/B/ST1/00554.
\bibliographystyle{plainnat}

\bibliography{Bibl}
\end{document}